\documentclass{amsart}
\usepackage[T1]{fontenc}
\usepackage{amsmath,amssymb,amsthm}
\usepackage{mathtools}
\usepackage[hypertexnames=false]{hyperref}
\usepackage{cleveref}
\usepackage{autonum}
\usepackage{booktabs}
\usepackage{longtable}
\usepackage{float}
\crefname{thm}{Theorem}{Theorem}
\crefname{lem}{Lemma}{Lemma}
\crefname{prop}{Proposition}{Proposition}
\crefname{cor}{Corollary}{Corollary}
\crefname{eg}{Example}{Example}
\crefname{definition}{Definition}{Definition}
\crefname{equation}{}{}
\crefname{section}{}{}
\newtheorem{thm}[subsubsection]{Theorem}
\newtheorem{lem}[subsubsection]{Lemma}
\newtheorem{prop}[subsubsection]{Proposition}
\newtheorem{cor}[subsubsection]{Corollary}
\newtheorem{eg}[subsubsection]{Example}
\newtheorem{definition}[subsubsection]{Definition}

\newtheorem{rem}[subsubsection]{Remark}
\newtheorem{prob}[subsubsection]{Problem}

\def\fr#1{\mathfrak{#1}}
\def\bb#1{\mathbb{#1}}

\def\ali#1{\begin{align}#1\end{align}}
\def\dis#1{\displaystyle{#1}}
\def\w{\wedge}
\def\bw{\bigwedge}
\def\o{\oplus}
\def\bo{\bigoplus}
\def\ang#1{\langle #1\rangle}
\def\dim#1{{\rm dim\ }#1}
\def\0{\{ 0\}}
\def\Ker#1{{\rm Ker}\ #1}
\def\Im#1{{\rm Im}\ #1}

\title[gradings associated with nilpotent fundamental groups]{Gradings on nilpotent Lie algebras associated with the nilpotent fundamental groups of smooth complex algebraic varieties}
\author{Taito Shimoji}
\date{}
\address{Department of Mathematics, Graduate School of Science, The University of Osaka, Osaka, Japan}
\email{u215629i@ecs.osaka-u.ac.jp}
\begin{document}
\maketitle
\begin{abstract}
Let $\Gamma$ be a lattice in a simply-connected nilpotent Lie group $N$ whose Lie algebra $\fr{n}$ is $p$-filiform. We show that  $\Gamma$ is either abelian or $2$-step nilpotent if $\Gamma$ is isomorphic to the fundamental group of a smooth complex algebraic variety. Moreover as an application of our result, we give a required condition of a lattice in a simply-connected nilpotent Lie group of dimension less than or equal to six to be isomorphic to the fundamental group of a smooth complex algebraic variety.
\end{abstract}
\section{Introduction}
The purpose of this paper is to give a class of finitely generated torsion-free nilpotent groups which cannot be isomorphic to the fundamental groups of smooth complex (possibly non-projective) algebraic varieties.
Let $N$ be a simply connected nilpotent Lie group and $\frak n$ the Lie algebra of $N$.
A lattice $\Gamma$ of  $N$ is a discrete subgroup such that the quotient $N/\Gamma$ is compact.
It is known that a lattice  $\Gamma$ of  $N$ is  a finitely generated torsion-free nilpotent group and conversely any finitely generated torsion-free nilpotent group  can be embedded into a unique simply connected nilpotent Lie group as a lattice (see \cite{Raghu}).
For a positive integer $p$, an $m$-dimensional nilpotent Lie algebra $\fr{n}$ is said to be \textit{$p$-filiform} Lie algebra if 
\ali{\dim C^i\fr{n}=m-i-p}
where $C^i\fr{n}$ is defined by $C^0\fr{n}:=\fr{n}$ and $C^i\fr{n}:=[\fr{n},C^{i-1}\fr{n}].$ In particular, the $1$-filiform Lie algebras are called the \textit{filiform} Lie algebras. 
The filiform is an important class of nilpotent Lie algebras whose lower central series decrease as slowly as possible. In \cite{Nil}, the filiform Lie algebras are first considered when classifying the complex nilpotent Lie algebras of dimension up to $7$. In particular, the $3$-dimensional filiform Lie algebra is only the $3$-dimensional Heisenberg Lie algebra.
On the other hand, the class of  filiform Lie algebras of higher dimensions is large and not classified completely (see \cite{Nil} and \cite{pfiliform}).
We obtain a large class of finitely generated torsion-free nilpotent groups which cannot be isomorphic to the fundamental groups of smooth complex algebraic varieties.
Our main theorem is the following:

\begin{thm}[\cref{Cor503}]
Let $\Gamma$ be a finitely generated torsion-free nilpotent group. Let $N$ be a simply-connected nilpotent Lie group such that $\Gamma$ is a lattice in $N$. Assume that the Lie algebra of $N$ $\fr{n}$ is $p$-filiform. If $\dim \fr{n}\geq p+3$, then $\Gamma$ is not isomorphic to the fundamental group of any  smooth complex algebraic variety. {\it i.e.} If the fundamental group of a smooth complex algebraic variety is isomorphic to a lattice in a simply connected nilpotent Lie group whose Lie algebra is $p$-filiform, then its nilpotency class is at most two.
\end{thm}
As a consequence, we conclude that if $\fr{n}$ is a filiform Lie algebra such that $\dim \fr{n}\geq 4$, then $\Gamma$ is not isomorphic to the fundamental group of any smooth complex algebraic variety. 
Some explicit examples are considered in this paper. For instance, the following shows such examples in \cref{filieg}.
\begin{eg}[\cref{filieg}]
For $n\geq 4$ and $m\geq 5$, the following groups are not isomorphic to the fundamental groups of smooth complex algebraic varieties.
\ali{
L_n&:=\left \langle x_1,\dots,x_n \middle\vert
\begin{array}{l}
[x_1, x_i]=x_{i+1}\ (1\leq i\leq n-1)\\
1=[x_1, x_n]=[x_j, x_k]\ (2\leq j<k\leq n)
\end{array}
\right \rangle \\
R_m&:= \left \langle x_1,\dots,x_m \middle\vert
\begin{array}{l}
[x_1, x_i]=x_{i+1}, [x_2, x_j]=x_{j+2}\\
(2\leq i\leq m-1, 3\leq j\leq m-2)\\
1=[x_1, x_m]=[x_2, x_{m-1}]=[x_2, x_m]=[x_k, x_l]\\
 (3\leq k<l\leq m)
\end{array}
\right \rangle \\
Q_{2m}&:= \left \langle x_1,\dots,x_m \middle\vert
\begin{array}{l}
[x_1, x_i]=x_{i+1}, [x_j, x_{m-j+1}]=x_m^{\varepsilon_j}\\
(2\leq i\leq m-1, 2\leq j\leq m-1,\varepsilon_j:=(-1)^{j+1})\\
1=[x_1, x_m]=[x_k, x_l]\ (2\leq k<l\leq m-1, k+l\neq m+1)
\end{array}
\right \rangle. 
}
\end{eg}
The main idea of proving the main theorem is an application of Morgan's mixed Hodge structures.
In \cite{Morgan}, Morgan constructs mixed Hodge structures on the Malcev completions of the fundamental groups of smooth complex algebraic varieties which are compatible with the canonical mixed Hodge structures constructed by Deligne in \cite{Del} by using rational homotopy theory. This shows the existence of certain gradings on the complexification $\frak{ n}_{\bb{C}}$ of the Lie algebra arising from a fundamental group of any smooth complex algebraic variety.
In case $\fr{n}$ is a $p$-filiform Lie algebra, we show that if such grading exists on $\frak{ n}_{\bb{C}}$ then  $\dim \fr{n}< p+3$.
Furthermore, in \cite{Morgan}, Morgan explains that the existence of his mixed Hodge structures can be criterion for finitely generated groups to be the fundamental groups of smooth complex algebraic varieties.
However, explicit classes of finitely generated torsion-free nilpotent groups which cannot be isomorphic to the fundamental groups have not been known so much.
We notice that Morgan's mixed Hodge structures are generalizations of the result in \cite{DGMS}.
Under the same setting as above, in case $M$ is projective, by the formality of a compact K\"ahler manifold in  \cite{DGMS}, we can say that the Lie algebra $\frak{ n}$ has a quadratic presentation {\it i.e.} It is the quotient of the free Lie algebra on its abelianization by an ideal generated in degree two. The existence of quadratic presentations on nilpotent Lie algebras are studied in \cite{CT}.

In \cref{5dim} and \cref{6dim}, we give a complete classification of nilpotent Lie algebras of dimension up to $6$ according to whether they admit gradings arising from Morgan's mixed Hodge structures. Eventually, we have the following theorem:
\begin{thm}[\cref{lowdim}]
Let $N$ be a simply connected nilpotent Lie group and $\fr{n}$ its Lie algebra such that $\dim\fr{n}\leq 6$. Let $\Gamma$ be a lattice of $N$. If $\fr{n}$ is isomorphic to one of the following Lie algebras:
\ali{
\text{filiform}&:L_{4,3}, L_{5,a}\ (a=6,7),\ L_{6,b}\ (14\leq b\leq 18)\\
\text{$2$-filiform}&:L_{4,3}\o\bb{C},\ L_{5,5},\ L_{5,a}\o\bb{C},\ L_{6,c}\ (11\leq c\leq 13)\\
\text{$3$-filiform}&:L_{4,3}\o\bb{C}^2,\ L_{5,5}\o\bb{C},\ L_{6,10}\\
\text{otherwise}&:L_{5,8},\ L_{6,19}(-1),\ L_{6,20},\ L_{6,23},\ L_{6,d}\ (25\leq d\leq 28)
}
then $\Gamma$ is not isomorphic to the fundamental group of any smooth complex algebraic variety. The symbol $L_{\bullet ,\bullet}$ is used in the list of low-dimensional nilpotent Lie algebras (see \cref{6dim}). 
\end{thm}

\subsection{Acknowledgements}
The author thanks to his supervisor Hisashi Kasuya for great discussions, valuable comments, introducing the Morgan's paper \cite{Morgan} to him, advising on the writing of the paper and pointing out the application of \cref{pfiliformMMHS} and thanks to his laboratory members: Valto Purho for letting him know the Raghunathan's text book \cite{Raghu} and Takashi Ono, Silveira Gomes Lucas Henrique for great discussions and helpful comments. This work was supported by JST SPRING, Grant Number JPMJSP2138. 

\section{Preliminaries and notations}
In this paper, we denote $\bb{K}$ by a field of characteristic $0.$
\subsection{$p$-filiform Lie algebras}
\begin{definition}
For an integer $p\geq 1$, an $m$-dimensional nilpotent Lie algebra $\fr{n}$ is said to be \textit{$p$-filiform} Lie algebra if 
\ali{\dim C^i\fr{n}=m-i-p.}
In particular, the $1$-filiform Lie algebras are called the \textit{filiform} Lie algebras.
\end{definition}
Note that the $m$-dimensional $p$-filiform Lie algebras are $(m-p)$-step nilpotent and $(m-1)$-filiform Lie algebras are abelian.
\begin{eg}\label{Heisenberg}
The Heisenberg Lie algebra $\fr{n}_3(\bb{K})$ is filiform because there is a basis $X_1,X_2$ and $X_3$ satisfying $[X_1,X_2]=X_3,[X_1,X_3]=0=[X_2,X_3]$ and $C^1\fr{n}_3(\bb{K})=\ang{X_3}, C^2\fr{n}_3(\bb{K})=\0.$
\end{eg}
In \cite{pfiliform}, there is a classification of the $m$-dimensional $(m-2)$ and $(m-3)$-filiform Lie algebra over $\bb{C}.$ From now on, if the Lie brackets are not explicitly given, they are understood to be zero.
\begin{eg}[\cite{pfiliform}]\label{m-2fili}
Let $q$ be an integer such that $\frac{m-3}{2}\leq q<\frac{m-1}{2}$. In dimension $m$, $m\geq 3$, there are $q$ pairwise non-isomorphic Lie algebras, denoted by $\fr{n}_m^q=\ang{X_0,X_1,X_2,Y_1,Y_2,\dots ,Y_{m-3}}$ whose laws are given by
\ali{&[X_0,X_1]=X_2\\
&[Y_{2k-1},Y_{2k}]=X_2\ \text{ for }1\leq k\leq q}.
\end{eg}
\begin{eg}[\cite{pfiliform}]\label{m-3fili}
In dimension $m$, $m\geq 5$, there are exactly $m-2$ pairwise non-isomorphic Lie algebras, denoted by $\fr{n}_m^{\bullet}=\ang{X_0,X_1,X_2,X_3,Y_1,\dots ,Y_{m-4}}$ whose laws are given by
\ali{\fr{n}_m^{2q-1}:\ &[X_0,X_i]=X_{i+1}\text{ for }i=1,2\ \ \ \ \text{ }1\leq q< \frac{m-2}{2}\\
&[Y_{2k-1},Y_{2k}]=X_3\ \text{ for }1\leq k\leq q\\\\
\fr{n}_m^{2q}:\ &[X_0,X_i]=X_{i+1}\text{ for }i=1,2\ \ \ \ \text{ }1\leq q< \frac{m-3}{2}\\
&[X_1,Y_{m-4}]=X_3\\
&[Y_{2k-1},Y_{2k}]=X_3\ \text{ for }1\leq k\leq q-1\\\\
\fr{n}_m^{m-2}:\ &[X_0,X_i]=X_{i+1}\text{ for }i=1,2\\
&[X_1,X_2]=Y_{m-4}}
\end{eg}
\subsection{Lie algebra cohomology}
The cohomology of a Lie algebra $\fr{g}$ is defined by the dual of the Lie bracket $d:=[,]^*:\fr{g}^*\rightarrow \fr{g}^*\w\fr{g}^*.$ We extend $d$ as the linear operator on the exterior algebra $\bw^{\bullet}\fr{g}$. Then we have $d\circ d=0$ and $d$ satisfies the Leibnitz rule. Thus $(\bw^{\bullet}\fr{g},d)$ is cochain complex. We call the cohomology of the complex $H^{\bullet}(\fr{g})=\Ker{d}/\Im{d}$ the cohomology of $\fr{g}.$
\begin{eg}\label{cohoeg}
Let us compute the first and second cohomologies of $(m-2)$-filiform Lie algebra $\fr{n}_6^2$. We take a basis as in \cref{m-2fili}, $\fr{n}_5^2=\ang{X_0,X_1,X_2,Y_1,Y_2}.$ Let $x_0,x_1,x_2,y_1,y_2$ be the dual basis. Then we have $[X_0,X_1]=X_2=[Y_1,Y_2],$  
$dx_0=0=dx_1=dy_{1}=dy_2$
and $dx_2=-x_0\w x_1-y_1\w y_2.$
Thus we get the first cohomology 
\ali{H^1(\fr{n}_5^2)=\ang{[x_0],[x_1],[y_1],[y_2]}.}
We compute the second cohomology. For $i=0,1$, $j=1,2$, we have $d(x_0\w x_1)=0=d(y_1\w y_2)=d(x_i\w y_j).$ By the Leibnitz rule, we have
\ali{&d(x_i\w x_2)=-x_i\w dx_2=x_i\w y_{1}\w y_{2}\\
&d(y_j\w x_2)=-y_j\w dx_2=y_j\w (x_0\w x_1)=x_0\w x_1\w y_j}
for i=$0$ or $1$ and $j=1,2.$ Thus it follows that 
\ali{H^2(\fr{n}_m^q)=\frac{\Ker{(d:\bw^2(\fr{n}_m^q)^*\rightarrow \bw^3(\fr{n}_m^q)^*)}}{\Im{(d:(\fr{n}_m^q)^*\rightarrow \bw^2(\fr{n}_m^q)^*)}}=\ang{[x_0\w x_1],[x_0\w y_1],[x_0\w y_2],[x_1\w y_1],[x_1\w y_2]}.}
\end{eg}
\subsection{Graded Lie algebras}
\begin{definition}\label{graded}
Let $\fr{g}$ be a $\bb{K}$-Lie algebra. A grading of $\fr{g}$ is a direct sum decomposition $\fr{g}=\dis{\bo_{i\in\bb{Z}}\fr{g}_i}$ such that 
\ali{[\fr{g}_i,\fr{g}_j]\subset \fr{g}_{i+j}\ \ {\rm for \ all}\ i,j\in\bb{Z}.} 
We call a Lie algebra equipped with a grading ``graded Lie algebra''. If a grading is indexed by natural( or negative) numbers, then the grading is said to be positive( or negative). A length of a grading $\fr{g}=\dis{\bo_{i\in\bb{Z}}\fr{g}_i}$ is defined by cardinal of a set 
\ali{\{ i\in\bb{Z}\mid \fr{g}_i\neq\0\}.} 
\end{definition}
By \cref{graded}, we see that finite-dimensional positive (or negative) graded Lie algebras are nilpotent.
\begin{lem}\label{central}
Let $\fr{n}$ be a Lie algebra and $\fr{n}=\fr{n}_{i_1}\o\cdots\o\fr{n}_{i_l}$ a negative grading. Then, for all $k\geq 1$, 
\ali{C^k\fr{n}\subset\dis{\bo_{i<i_k}\fr{n}_i}.} 
\end{lem}
\begin{proof}
For $k=1$, by the definition of graded Lie algebra, we see that
\ali{C^1\fr{n}&=[\fr{n}_{i_1}\o\cdots\o\fr{n}_{i_l},\fr{n}_{i_1}\o\cdots\o\fr{n}_{i_l}]\subset \dis{\bo_{i<2i_1}\fr{n}_i}\subset \dis{\bo_{i<i_1}\fr{n}_i}.} 
We assume that $C^{j}\fr{n}\subset\dis{\bo_{i<i_{j}}\fr{n}_i}$ for $j=1,2,\dots k-1$. Thus, by the definition of the lower central series, we have
\ali{C^k\fr{n}&=[\fr{n}, C^{k-1}\fr{n}].}
By assumption and the definition of the graded Lie algebra, it follows that
\ali{[\fr{n}, C^{k-1}\fr{n}]&\subset [\fr{n}_{i_1}\o\cdots\o\fr{n}_{i_l},\dis{\bo_{i<i_{k-1}}\fr{n}_i} ]=[\fr{n}_{i_1}\o\cdots\o\fr{n}_{i_l},\fr{n}_{i_k}\o\cdots\o\fr{n}_{i_l}]\subset \dis{\bo_{i\leq i_{k}+i_1}\fr{n}_i}\subset \dis{\bo_{i<i_{k}}\fr{n}_i}
}
\end{proof}
\begin{lem}\label{gradingstep}
A length of a negative grading on a $k$-step nilpotent Lie algebras is greater or equal to than $k$. 
\end{lem}
\begin{proof}
Let $\fr{n}$ be a $k$-step nilpotent Lie algebra and $\fr{n}=\fr{n}_{i_1}\o\cdots\o\fr{n}_{i_l}$ a negative grading. We prove $l\geq k$. We assume that $l<k$. Note that $C^l\fr{n}\neq\0$ because $\fr{n}$ is $k$-step and $l\leq k-1$. By \cref{central}, we have
\ali{\0 \neq C^l\fr{n}\subset \dis{\bo_{i<i_{l}}\fr{n}_i}=\0.
}
This is a contradiction. Therefore $l\geq k$ holds.
\end{proof}
\begin{rem}
A graded Lie algebra $\fr{g}=\bo_{i\in \bb{Z}}\fr{g}_i$ induces the grading on $\bw^j\fr{g}^*=\bo_{k\in \bb{Z}}C_k^j$ where $C_k^j:=\bo_{i_1+\cdots +i_j=k}\fr{g}_{-i_1}^*\w\cdots \w \fr{g}_{-i_j}^*$, and $\fr{g}_{-i_1}^*\w\cdots \w \fr{g}_{-i_j}^*$ is defined by a subvector space generated by the set $\{ f_{i_1}\w\cdots \w f_{i_j}\mid f_{i_l}\in \fr{g}_{-i_l}^*\}.$ Then we see that $d(C_k^j) \subset C_k^{j+1}.$ Thus the graded Lie algebra induces the grading on the cohomology $H^j(\fr{g})=\bo_{i\in\bb{Z}}H_k^j$ where $H_k^j:=\dis{\frac{\Ker{(d:C_k^j\rightarrow C_k^{j+1})}}{\Im{(d:C_k^{j-1}\rightarrow C_k^{j})}}}.$
\end{rem}
Given a basis for each homogeneous component of a graded vector space, for instance, we denote the degree $i$ component, whose basis consists of $X_1\dots X_n$ by $\ang{X_1,\dots ,X_n}_i.$ 
\begin{eg}\label{n-2fili}
We define a grading on $\fr{n}_5^2=\ang{X_0,X_1,X_2,Y_1,Y_2}$ in \cref{cohoeg} to be 
\ali{\fr{n}_5^2=\ang{X_0,X_1,Y_1,Y_2}_{-1}\o\ang{X_2}_{-2}.}
Let $x_0,x_1,x_2,y_1$ and $y_2$ are the dual basis. Then the gradings on the first and second cohomologies are
\ali{H^1(\fr{n}_5^2)=H_1^1=\ang{[x_0],[x_1],[y_1],[y_2]}_1}
and
\ali{H^2(\fr{n}_5^2)=H_2^2=\ang{[x_0\w x_1],[x_0\w y_1],[x_0\w y_2],[x_1\w y_1],[x_1\w y_2]}_2.}
\end{eg}
\section{Mixed Hodge structures}\label{6}
This section is based on \cite{Morgan}.
\begin{definition}
An $\bb{R}$-Hodge filtration of weight $n$ on an $\bb{R}$-vector space $V$ is a finite decreasing filtration $\{ F^p(V_{\bb{C}})\}_{p\in\bb{Z}}$ on $V_{\bb{C}}:=V\otimes \bb{C}$ such that
\ali{V_{\bb{C}}=F^p(V_{\bb{C}})\o \overline{F^{n+1-p}(V_{\bb{C}})}}
for all integer $p$.
\end{definition}
\begin{definition}
A pair $(V,W,F)$ is an $\bb{R}$-mixed Hodge structure on an $\bb{R}$-vector space $V$ if the following conditions hold:\\
\indent
$(1)$ $W=\{W_k(V)\}_{k\in\bb{Z}}$ is an increasing filtration on $V$ bounded below.\\
\indent
$(2)$ $F=\{ F^p(V_{\bb{C}})\}_{p\in\bb{Z}}$ is a finite decreasing filtration on $V_{\bb{C}}.$\\
\indent
$(3)\label{Hodge}$ The decreasing filtration on $Gr_k^W(V_{\bb{C}}):=\dis{W_k(V)\over W_{k-1}(V)} \otimes \bb{C}$, \ $F^p(Gr_k^W(V_{\bb{C}})):=\dis{\frac{F^p(V_{\bb{C}})\cap (W_k(V)\otimes \bb{C})}{W_{k-1}(V)\otimes\bb{C}}}$\\
\hspace{1cm}is an $\bb{R}$-Hodge filtration of weight $k$ on $\dis{Gr_k^W(V):=\frac{W_k(V) }{W_{k-1}(V)}}$ for all integer $k$. 
\end{definition}
\begin{prop}{\cite{Morgan}}\label{MHS}
Let $V$ be an $\bb{R}$-vector space. If $V$ admits an $\bb{R}$-mixed Hodge structure $(V,W,F)$, then $V_{\bb{C}}$ admits a bigrading $V_{\bb{C}}=\dis{\bo_{p,q\in\bb{Z}} V^{p,q}}$ such that 
\ali{\overline{V^{p,q}}=V^{q,p}\ {\rm mod}\ \dis{\bo_{s+t<p+q}V^{s,t}},}
\ali{W_i(V_{\bb{C}})=\dis{\bo_{p+q\leq i}V^{p,q}},\ {\rm and}\ F^i(V_{\bb{C}})=\dis{\bo_{p\geq i, q\in\bb{Z}}V^{p,q}}.} 
\end{prop}
Let $(V,W,F)$ be an $\bb{R}$-mixed Hodge structure and $V_{\bb{C}}=\dis{\bo_{p,q\in\bb{Z}} V^{p,q}}$ the bigrading in \cref{MHS}. Let $V_i:=\dis{\bo_{p+q=i}V^{p,q}}$. We obtain the grading $V_{\bb{C}}=\dis{\bo_{i\in\bb{Z}}V_i}$ such that $\dim V_i$ is even for odd integer $i$.

\begin{eg}[\cite{Del}]\label{Deligne}
Let $M$ be a any smooth complex algebraic variety.
The $k$-th real cohomology $ H^{k}(M,\bb{R})$ admits a canonical $\bb{R}$-mixed Hodge structure.
For  the  bigrading $H^k(M,\bb{R})_{ \bb{C}}=\dis{\bo_{i\geq 1}H_{p,q}^k}$ induced by such mixed Hodge structure as in Proposition \ref{MHS}, we have $H^1(M,\bb{R})_{ \bb{C}}=H_{1,0}^1\oplus H_{0,1}^1\oplus H_{1,1}^1$ and $H^2(M,\bb{R})_{ \bb{C}}=H_{2,0}^2\oplus H_{1,1}^2\oplus  H_{0,2}^2\oplus H_{2,1}^2\oplus H_{1,2}^2\oplus  H_{2,2}^2$.

\end{eg}

\subsection{Morgan's mixed Hodge structures}
Let $M$ be a any smooth complex algebraic variety and $x\in M$ based point. 
We assume that  the fundamental group $\pi_{1}(M,x)$ is a lattice in a simply connected nilpotent Lie group $N$.
In this case, the Lie algebra $\frak n$ of $N$ is the Lie algebra of  the Malcev completion of $\pi_1(M,x)$ and the differential graded algebra $\bigwedge^{\bullet} \frak n^{*}$ is the $1$-minimal model of $M$  in the sense of Sullivan (\cite{Sul}).
Thus, we have an isomorphism $H^{1}({\frak n})\cong H^{1}(M,\bb{R})$ and an injection $H^{2}({\frak n})\hookrightarrow H^{2}(M,\bb{R})$.
In \cite{Morgan}, Morgan shows the existence of  a mixed Hodge structure of the Lie algebra of  the Malcev completion of the fundamental group of a any smooth complex algebraic variety which is an extension of the canonical mixed Hodge structure as Example \ref{Deligne}.
In our assumption, we can say that  the Lie algebra $\frak n$ admits an $\bb{R}$-mixed Hodge structure such that 
the induced  $\bb{R}$-mixed Hodge structure on $H^{1}({\frak n})$ is isomorphic to the canonical $\bb{R}$-mixed Hodge structure on 
$ H^{1}(M,\bb{R})$ and  the induced  $\bb{R}$-mixed Hodge structure on $H^{2}({\frak n})$ is a $\bb{R}$-mixed Hodge substructure of  the canonical $\bb{R}$-mixed Hodge structure on 
$ H^{2}(M,\bb{R})$.
Taking the total grading of the bigrading   induced by such mixed Hodge structure as in Proposition \ref{MHS}, we obtain the following statement.

\begin{thm}[\cite{Morgan}]\label{WH}
Let $M$ be a any smooth complex algebraic variety, $x\in M$ based point.
 We assume that  the fundamental group $\pi_{1}(M,x)$ is a lattice in a simply connected nilpotent Lie group $N$.
 Then there exists a grading $\fr{n}_{\bb{C}}=\dis{\bo_{i\leq -1}\fr{n}_i}$ such that:
\\\ \indent
{\bf (W):}\label{W} For the induced grading $H^*(\fr{n}_{\bb{C}})=\dis{\bo_{i\geq 1}H_i^*}$, we have
\ali{H^1(\fr{n}_{\bb{C}})=H_1^1\o H_2^1}
\\\ \indent
and
\ali{H^2(\fr{n}_{\bb{C}})=H_2^2\o H_3^2\o H_4^2.} 
\\\ \indent
{\bf (H):}\label{H} For any odd integer $k$, $\dim{\fr{n}_k}$ and $\dim{H_k^j}$ are even. 
\end{thm}
\begin{cor}\label{fundamentallattice}
Let $N$ be a simply connected nilpotent Lie group and $\fr{n}$ its Lie algebra. Let $\Gamma$ be a lattice of $N$. If $\fr{n}$ does not admit a grading satisfying {\bf (W)} and {\bf (H)}, then $\Gamma$ is not isomorphic to the fundamental group of any smooth complex algebraic variety. 
\end{cor}
\begin{proof}
Assume that $\Gamma$ is isomorphic to the fundamental group of a smooth algebraic variety. By \cite{Malcev}, the $\bb{R}$-Malcev completion of $\Gamma$ is isomorphic to $\fr{n}$. Thus, by \cref{WH}, we see that $\fr{n}$ admits a grading satisfying {\bf (W)} and {\bf (H)}. This contradicts the assumption.
\end{proof}
\begin{eg}\label{abelians}
Let $\bb{C}$ be the $1$-dimensional abelian Lie algebra. Take a non-zero element $X$. Then the grading $\bb{C}=\ang{X}_{-2}$ satisfies the conditions {\bf (W)} and {\bf (H)}. In the case of the $2$-dimensional abelian Lie algebra $\bb{C}^2$, we may set $\bb{C}^2=\ang{X,Y}_{-1}$ where $X$ and $Y$ form a basis of $\bb{C}^2.$  
\end{eg}
\begin{eg}
By \cref{n-2fili}, the grading of $\fr{n}_5^2$, $\fr{n}_5^2=\ang{X_0,X_1,Y_1,Y_2}_{-1}\o\ang{X_2}_{-2}$ satisfies the conditions.
\end{eg}
\begin{prop}\label{trivialextension}
Let $\fr{n}$ be a nilpotent Lie algebra which admits a grading satisfying {\bf (W)} and {\bf (H)}. Then any trivial extension $\fr{n}\o\bb{C}^m$ also admits such a grading.
\end{prop}
\begin{proof}
Let $\fr{n}=\bo_{i\leq -1}\fr{n}_i$ be a grading satisfying the conditions {\bf (W)} and {\bf (H)}, and the induced grading $H^*(\fr{n})=\bo_{i\geq 1}H_i^*$. We define $({\fr{n}\o\bb{C}^m})_{-2}:=\fr{n}_{-2}\o\bb{C}^m$ and $({\fr{n}\o\bb{C}^m})_{i}:=\fr{n_i}$ for other $i$. Since the differentiation of the dual basis for any basis of $\bb{C}^m$ are all zero, we have 
\ali{d(x\w y)=dx\w y}
for all $x\in \bw^{\bullet}\fr{n}^*$ and $y\in \bw^{\bullet}(\bb{C}^m)^*.$ Hence we have $dx=0$ if and only if $d(x\w y)=0$ for any $x\in \bw^{\bullet}\fr{n}^*$ and $y\in \bw^{\bullet}(\bb{C}^m)^*.$ Thus we see that 
\ali{H^1(\fr{n}\o\bb{C}^m)=H_1^1\o {H_2^1}',\ H^2(\fr{n}\o\bb{C}^m)=H_2^2\o {H_3^2}'\o{H_4^2}'}
where ${H_2^1}'=H_2^2\o (\bb{C}^m)^*$, ${H_3^2}'=H_3^2\o\ang{[x\w y]\mid x\in (\fr{n}_{-1})^*,y\in (\bb{C}^m)^*}.$ and ${H_4^2}'=H_4^2\o\ang{[y_i\w y_j]\mid y_i,y_j\in (\bb{C}^m)^*, 1\leq i<j\leq m}$ Thus the grading satisfies the condition {\bf (W)}. Since \ali{\dim {{H_3^2}'}=m\cdot\dim H_1^1 +\dim H_2^3\in 2\bb{Z}.} Therefore the condition {\bf (H)} holds.  
\end{proof}
\begin{eg}
Let us consider the grading on $\fr{n}_3(\bb{C}).$ Take the basis $X_1,X_2$ and $X_3$ same as \cref{Heisenberg}, we define the grading on $\fr{n}_3(\bb{C})$ to be $\fr{n}_3(\bb{C})=\ang{X_1,X_2}_{-1}\o\ang{X_3}_{-2}.$ Then the induced grading on the first and second cohomologies of $\fr{n}_3(\bb{C})$ are
\ali{H^1(\fr{n}_3(\bb{C}))=H_1^1=\ang{[x_1],[x_2]}_{1},\ 
H^2(\fr{n}_3(\bb{C}))=H_2^3=\ang{[x_1\w x_3],[x_2\w x_3]}_{3}}
where $x_1,x_2$ and $x_3$ are the dual basis. Consider the direct sum decomposition $\fr{n}_3(\bb{C})\o\bb{C}^m$, take a basis $Y_1,\dots, Y_m\in \bb{C}^m$ and let $y_1,\dots, y_m$ be the dual basis. Let $\fr{n}_3(\bb{C})\o\bb{C}^m=\ang{X_1,X_2}_{-1}\o\ang{X_3,Y_1,\dots,Y_m}_{-2}$ be the grading. Since $dy_j=0$ and $d(x_i\w y_j)=0$, and $d(x_3\w y_j)=-x_1\w x_2\w y_j$ for $i=1,2$ and $1\leq j\leq m$, we have
\ali{H^1(\fr{n}_3(\bb{C})\o\bb{C}^m)=\ang{[x_1],[x_2]}_1\o\ang{[y_1],\dots, [y_m]}_2}
and
\ali{H^2(\fr{n}_3(\bb{C})\o\bb{C}^m)=\ang{[x_1\w x_3],[x_2,\w x_3],[x_i\w y_j]\mid i=1,2,\ 1\leq j\leq m}_{3}.}
Thus the grading satisfies the conditions {\bf (W)} and {\bf (H)}.
\end{eg}
\section{Main Theorem}\label{9}
\begin{prop}\label{mainthm}
Let $\fr{n}$ be a filiform $\bb{C}$-Lie algebra satisfying the conditions {\bf (W)} and {\bf (H)} (in \cref{W}). Then we have $\dim \fr{n}<4.$
\end{prop}
\begin{proof}
Let $\fr{n}=\fr{n}_{i_1}\o\cdots\o\fr{n}_{i_k}$ be a grading such that the conditions {\bf (W)} and {\bf (H)} hold. To prove the proposition, we assume that $\dim\fr{n}=n\geq 4$. Since $\fr{n}$ is $(n-1)$-step nilpotent, then we have $k=n-1$ or $n$. This is because $\fr{n}$ is filiform and follows by \cref{gradingstep}.

Assume that $k=n$. We have $\dim\fr{n}_{i_j}=1$ for all $1\leq j\leq n$. Since $\dim C^1\fr{n}=n-2$, we have $C^1\fr{n}=\fr{n}_{i_3}\o\cdots\o\fr{n}_{i_n}$. Thus $H^1(\fr{n})\supset \fr{n}_{i_1}^*\o\fr{n}_{i_2}^*.$ We see that the dimensions of both sides are $2$. By {\bf (W)}, it follows that $H_1^1\o H_2^1=H^1(\fr{n})=\fr{n}_{i_1}^*\o\fr{n}_{i_2}^*$. Thus $\dim\fr{n}_{-1}=1\neq even$. This contradicts {\bf (H)}. Thus $k=n-1$.

We show that $\dim\fr{n}_{i_1}=2$. Now that $k=n-1$, we have $\dim\fr{n}_{i_j}=2$ for some $1\leq j\leq n-1$. If $j\neq 1$, then $\fr{n}$ can be written by 
\ali{\fr{n}&=\fr{n}_{i_1}\o\cdots\o\fr{n}_{i_{n-1}}\\
&=\ang{Y_1}_{i_1}\o\fr{n}_{i_2}\o\cdots .}
Hence we have $n-2=\dim C^1\fr{n}=\dim(\fr{n}_{i_3}\o\cdots)=n-3$ or $n-2.$ Thus $\dim\fr{n}_{i_2}\neq 2$ and $j>2$. Then, for the same reason when $k=n$, we have $\dim\fr{n}_{-1}=1\neq even$. This is a contradiction. Thus $\dim{n}_{i_1}=2$.

Let $\fr{n}=\ang{Y_1,Y_2}_{i_1}\o\ang{Y_3}_{i_2}\o\ang{Y_4}_{i_3}\o\cdots$ and $\fr{n}^*=\ang{y_1,y_2}_{-i_1}\o\ang{y_3}_{-i_2}\o\ang{y_4}_{-i_3}\o\cdots$. Since $dy_1=dy_2=0$ and $\dim H^1(\fr{n})=2$, we have 
\ali{\ang{y_1,y_2}_{-i_1}=H^1(\fr{n})=H_1^1\o H_2^1.}
Thus $i_1=-1$ or $-2$. Assume $i_1=-1$. Write $dy_3=-ay_1\w y_2\ (a\in\bb{C})$. Since $y_1,y_2$ and $y_3$ are linearly independent, $a\neq 0$. Thus $\fr{n}_{-2}\ni [Y_1,Y_2]=aY_3(\neq 0) \in\fr{n}_{i_2}$ then $i_2=-2$. There exist constants $a', b', c'\in\bb{C}$ such that 
\ali{dy_4=-a'y_1\w y_2-b'y_1\w y_3 -c'y_2\w y_3.}
Equivalently, we have
\ali{\fr{n}_{i_3}\ni Y_4=a'[Y_1,Y_2]+b'[Y_1,Y_3]+c'[Y_2,Y_3]\in \fr{n}_{-2}\o\fr{n}_{-3}.}
Thus $\fr{n}_{i_3}\cap (\fr{n}_{-2}\o\fr{n}_{-3})\neq \0$. It follows that $\fr{n}_{i_3}\cap\fr{n}_{-2}\neq\0$ or $\fr{n}_{i_3}\cap\fr{n}_{-3}\neq\0$. Since $y_3$ and $y_4$ are linearly independent, we have $\fr{n}_{i_3}\cap\fr{n}_{-2}=\0$. Thus $i_3=-3$. This contradicts {\bf (H)}. Thus $i_1=-2$. For the same reason when $i_1=1$, we can write 
\ali{\fr{n}=\ang{Y_1,Y_2}_{-2}\o\ang{Y_3}_{-4}\o\ang{Y_4}_{-6}\o\cdots,} 
\ali{\fr{n}^*=\ang{y_1,y_2}_{2}\o\ang{y_3}_{4}\o\ang{y_4}_{6}\o\o\cdots,}
and 
\ali{\bw^2\fr{n}^*=\ang{y_1\w y_2}_{4}\o\ang{y_1\w y_3,y_2\w y_3}_{6}\o\cdots.}
Since $dy_3=-ay_1\w y_2$ $(a\neq 0)$, $y_1\w y_2\in {\rm Im}(d:\fr{n}^*\rightarrow \bw^2\fr{n}^*)$ and the condition {\bf (W)}, we have $H^2(\fr{n})=H_4^2=\0$. Since $d(y_1\w y_3)=d(y_2\w y_3)=0,$ we see that ${\rm Ker}:(d:C_6^2\rightarrow C_6^3)=\ang{y_1\w y_3, y_2\w y_3}.$ Moreover, since $\dim {\rm Im}(d:C_6^1\rightarrow C_6^2)\leq 1,$ we have $\dim H_6^2\geq 1$. This contradicts that $H^2(\fr{n})=\0.$
\end{proof}
\begin{eg}\label{Heisengrad}
For $n\geq 3,$ we define the Lie algebra $\fr{n}_n=\ang{X_1,X_2,X_3,\dots,X_n}$ to be $[X_1,X_i]=X_{i+1}.$ If $n=3$, then $\fr{n}_3$ admits a grading satisfying the conditions. One can take $\fr{n}_3=\ang{X_1,X_2}_{-1}\o \ang{X_3}_{-2}$ in this case. The grading on the first and second cohomologies are
\ali{H^1(\fr{n}_3)=\ang{[x_1],[x_2]}_1}
and
\ali{\ H^2(\fr{n}_3)=\ang{[x_1\w x_3],[x_2\w x_3]}_3}
where $x_1,x_2$ and $x_3$ are the dual basis.
 Moreover, by the filiformness of $\fr{n}_n$ and \cref{mainthm}, the Lie algebra $\fr{n}_n$ admits a grading satisfying the conditions {\bf (W)} and {\bf (H)} if and only if $n=3$.
\end{eg}
\begin{prop}\label{10.5}
For an integer $p\geq 2$, we assume that $\fr{n}$ is a $p$-filiform $\bb{C}-$Lie algebra dimensional $m\geq 5$ and $m\geq p+3$. Let $\fr{n}=\bo_{i\leq -1}\fr{n}_i$ be a grading satisfying {\bf (W)} and {\bf(H)}. For the subspaces $\fr{n}_{-1}$ and $\fr{n}_{-2}$, we have
\begin{align}
\fr{n}_{-1}\neq \0, \fr{n}_{-2}\neq \0 \text{ and } \dim \fr{n}_{-2}^*=\dim H_2^1+1. 
\end{align}
\end{prop}
\begin{proof}
Write $\fr{n}=\fr{n}_{i_1}\o \cdots \o\fr{n}_{i_k}$. Since $\fr{n}$ is $(m-p)$-step and by the \cref{gradingstep}, we have $k\geq m-p$. By the condition {\bf (W)}, $i_1=-1$ or $i_1=-2$.

We first assume that $i_1=-2$. By {\bf (W)} and $p-$filiformness of $\fr{n}$, we have \ali{\dim \fr{n}_{-2}=\dim H^1(\fr{n})=\dim \fr{n}-\dim C^1\fr{n}=m-(m-1-p)=p+1.\notag}
Thus we can write $\fr{n}=\ang{X_1,\dots, X_{p+1}}_{-2}\o\fr{n}_{i_2}\o\cdots\o \fr{n}_{i_k}$. Since the number of the components of $\fr{n}_{i_2}\o\cdots\o \fr{n}_{i_k}$, that is, $k-1\leq \dim (\fr{n}_{i_2}\o\cdots\o \fr{n}_{i_k})=m-(p+1)$ and $m-p\leq k$, it follows that $k=m-p$ and $\dim\fr{n}_{i_j}=1$ for all $2\leq j\leq m-p.$ Thus we can write \ali{\fr{n}=\ang{X_1,\dots, X_{p+1}}_{-2}\o\ang{X_{p+2}}_{i_2}\o\ang{X_{p+3}}_{i_3}\o\cdots.}Let $x_i$ $(1\leq i\leq m)$ be the dual basis of $X_i$. Since $X_{p+2},X_{p+3}\in C^1\fr{n}$, we have $dx_{p+2}, dx_{p+3}\neq 0$. Thus we have $i_2=-4$ and $i_3=-6$. For $1\leq i\leq p+1$, we see that  $d(x_i\w x_{p+2})=0$. Then \ali{\dim {\rm Ker}(d:C_6^2\rightarrow C_6^3)=p+1,\notag\\
\dim {\rm Im}(d:\fr{n}_{-6}^*\rightarrow C_6^2)=1,\notag \\
\dim H_6^2=p+1-1=p\geq 2.}
This contradicts {\bf (W)}. Thus, we conclude that $i_1=-1$.

By {\bf (H)}, we can write the grading as $\fr{n}=\ang{X_1,\dots, X_{2l}}_{-1}\o\fr{n}_{i_2}\o\cdots\o\fr{n}_{i_k}$ for some integer $l$. If $\dim\fr{n}_{-1}=p+1$, then, by {\bf (H)}, we have $p+1\in 2\bb{Z}$ and $\dim (\fr{n}_{i_2}\o\cdots\o\fr{n}_{i_k})=m-(p+1)=m-p-1=\dim C^1\fr{n}$. For the same reason as the case when $i_1=-2$, we see that $k=m-p$, $\dim\fr{n}_{i_j}=1$ for all $2\leq j\leq m-p$, $i_2=-2$ and $i_3=-3$, which contradicts the assumption {\bf (H)}. Thus $\dim\fr{n}_{-1}<p+1$ and there exists a non-zero element $x\in H^1(\fr{n})$ such that $x\notin H_1^1$. Thus we have $i_2=-2$ and we can write\ali{\fr{n}=\ang{X_1,\dots, X_{2l}}_{-1}\o\ang{X_{2l+1},\dots, X_{p+1},\dots ,X_{p+1+q}}_{-2}\o\fr{n}_{i_3}\o\cdots\o\fr{n}_{i_k}.}
We show that $q=1$. By the proof of \cref{gradingstep}, we have $m-2-p=\dim C^2\fr{n}\leq \dim (\fr{n}_{i_3}\o\cdots\o\fr{n}_{i_k})=m-(p+1+q)$. We thus get $q\leq 1$.

Assume that $q=0$. {\it i.e.} we assume that $\fr{n}=\ang{X_1,\dots, X_{2l}}_{-1}\o\ang{X_{2l+1},\dots, X_{p+1}}_{-2}\o\fr{n}_{i_3}\o\cdots\o\fr{n}_{i_k}.$ Combining with $\dim C^2\fr{n}=m-2-p\leq k-2\leq \dim (\fr{n}_{i_3}\o\cdots\o\fr{n}_{i_k})=m-(p+1)$, we have $k=m-p$ or $m-p+1$. Thus we see that $\dim \fr{n}_{i_3}=1$ or $2$. Write $\fr{n}_{i_3}=\ang{X_{p+2}}$ (or $\ang{X_{p+2},X_{p+3}}$). Let $x_i$ $(1\leq i\leq m)$ be the dual basis of $X_i$. Then we see that $d(x_j\w x_{p+2})=0$ (and $d(x_k\w x_{p+3})=0$) for all $2l+1\leq j\leq p+1$ (and $2l+1\leq k\leq p+1$). By {\bf (W)}, we have $\dim {\rm Ker}(d:C_6^2\rightarrow C_6^2)=p+1-2l=\dim {\rm Im}(d:\fr{n}_{-6}^*\rightarrow C_6^2)\leq 1$. Then we see that $p\leq 2l<p+1$, which implies $2l=p$. In the same reason, we also have $2l=p$ if $\dim \fr{n}_{i_3}=2$. By {\bf (H)}, we see that $i_3=-3$ and we can write
\ali{\fr{n}=\ang{X_1,\dots, X_p}_{-1}\o\ang{X_{p+1}}_{-2}\o\ang{X_{p+2},X_{p+3}}_{-3}\o\cdots\o\fr{n}_{i_k}.}
Since $d(x_{p+1}\w x_{p+2})=0=d(x_{p+1}\w x_{p+3})$ and $\dim {\rm Im}(d:\fr{n}_{-5}^*\rightarrow C_5^2)\leq 1$, we have $H_5^2\neq \0$. This contradicts {\bf (W)}. Thus $q=1$. Therefore $\dim \fr{n}_{-2}=\dim H_1^2+1$. 
\end{proof}
\begin{thm}\label{pfiliformMMHS}
Let $\fr{n}$ be a $p$-filiform Lie algebra which admits a grading $\fr{n}=\bo_{i\leq -1}\fr{n}_{i}$ satisfying the conditions {\bf (W)} and {\bf (H)} if and only if
\ali{\dim \fr{n}<p+3.}
Equivalently, $p$-filiform Lie algebra $\fr{n}$ admits a grading $\fr{n}=\bo_{i\leq -1}\fr{n}_i$ satisfying the conditions {\bf (W)} and {\bf (H)} if and only if $\fr{n}$ is either abelian or $2$-step nilpotent Lie algebra.
\end{thm}
\begin{proof}
We assume that $\dim \fr{n}=m\geq p+3.$ By \cref{mainthm}, it suffices to consider the cases where $p\geq 2.$ Thus, by \cref{10.5}, we conclude that $\dim \fr{n}_{-2}=\dim H_2^1+1.$ Write
\ali{\fr{n}=\ang{X_1,\dots X_{2l}}_{-1}\o\ang{X_{2l+1},\dots, X_{p+2}}_{-2}\o\fr{n}_{i_3}\o\cdots\o\fr{n}_{i_k}} 
where $X_{p+2}\notin C^1\fr{n}$ and $X_{p+2}\in C^1\fr{n}$. Since $m-p-2\leq k-2\leq \dim (\fr{n}_{i_3}\o\cdots \o\fr{n}_{i_k})=m-(p+2),$ we have $k=m-p$ and $\dim\fr{n}_{i_3}=1.$ If there exists a pair $X_i$ and $X_j$ such that $1\leq i\leq 2l,$ $2l+1\leq j\leq p+1$ and $[X_i,X_j]\neq 0,$ we have $i_3=-3.$ This contradicts {\bf (H)}. Thus we have $[X_i,X_j]=0\ \text{for }1\leq i\leq 2l,2l+1\leq j\leq p+1.$ Thus we see that $i_3=-4$. Since $X_{p+2}\in C^1\fr{n}$, there exist $X_s,X_t\in \fr{n}_{-1}$ such that $[X_s,X_t]=C_{st}^{p+2}X_{p+2}\neq 0.$ Thus we can write
\ali{\fr{n}=\ang{X_1,\dots X_{2l}}_{-1}\o\ang{X_{2l+1},\dots, X_{p+2}}_{-2}\o\ang{X_{p+3}}_{-4}\o\cdots.} Since $C^2\fr{n}=\fr{n}_{i_3}\o\cdots \o\fr{n}_{i_{m-p}}$ and $C^3\fr{n}=\fr{n}_{i_4}\o\cdots \o\fr{n}_{i_{m-p}}$, we have $X_{p+3}\in C^2\fr{n}$ and $X_{p+3}\notin C^3\fr{n}$. Thus we can take an element $X_k\in\fr{n}_{-2}$ such that $[X_k,X_{p+2}]=C_{kp+2}^{p+3}X_{p+3}\neq 0$ for $2l+1\leq k\leq p+1.$ By the Jacobi identity, we have 
\ali{0=[X_s,[X_t, X_k]]&=[[X_s,X_t],X_k]+[X_t,[X_s,X_k]]\\
&=C_{st}^{p+2}C_{kp+2}^{p+3}\neq 0.}
This is a contradiction. Therefore $m<p+3$. Thus $p=\dim\fr{n}-1$ or $\dim\fr{n}-2$. By the definition of $p$-filiform Lie algebras, it follows that $\fr{n}$ is either abelian or $2$-step nilpotent. Conversely, we construct a grading satisfying {\bf (W)} and {\bf (H)} on a Lie algebra $\fr{n}_m^q$ for all $m$ and $q$ such that $\frac{m-3}{2}\leq q<\frac{m-1}{2}$(see \cref{m-2fili} or \cite{pfiliform}). Fix such integers $m$ and $q$. We define the grading on $\fr{n}_m^q$ by
\ali{\fr{n}_m^q=
\ang{X_0,X_1,Y_1,\dots, Y_{2q-3},Y_{2q-2}}_{-1}\o\ang{Y_{2q-1},Y_{2q},\dots, Y_{m-3},X_2}_{-2}}
Then the induced first and second cohomologies are as follows:
\ali{
H^1(\fr{n}_m^q)&=H_1^1\o H_2^1\\
H^2(\fr{n}_m^q)&=H_2^1\o H_3^2\o H_4^2
}
where the spaces $H_{*}^1$ and $H_{*}^2$ are given by 
\ali{
H_1^1&=\ang{[x_i],[y_j]\mid i=0,1,1\leq j\leq 2q-2}\\
H_2^1&=\ang{[y_i]\mid 2q-1\leq i\leq m-3}
}
and
\ali{
H_2^2&=\ang{x_0\w x_1,x_i\w y_j,y_k\w y_l\mid i=0,1, 1\leq j\leq 2q-2,1\leq k<l\leq 2q-2}/\ang{dx_2}\\
H_3^2&=\ang{[x_i\w y_j],[y_k\w y_j]\mid i=0,1,2q-1\leq j\leq m-3, 1\leq k\leq 2q-2}\\
H_4^2&=\ang{[y_i\w y_j]\mid 2q-1\leq i<j\leq m-3}
.}
Thus we see that $\dim H_1^1=2+2q-2=2q\in 2\bb{Z}$ and $\dim H_3^2=2(m-3-(2q-2))+(2q-2)(m-3-(2q-2))=(2+2q-2)(m-2q-1)=2q(m-2q-1)\in 2\bb{Z}.$
Therefore the grading satisfies {\bf (W)} and {\bf (H)}. For the cases that $\dim \fr{n}=1,2,3$, we give in \cref{abelians},\cref{Heisengrad} and it follows from \cref{trivialextension}. 
\end{proof}
\begin{cor}\label{Cor503}
Let $N$ be a simply connected nilpotent Lie group and $\fr{n}$ its Lie algebra. Let $\Gamma$ be a lattice of $N$. If $\fr{n}$ is a $p$-filiform Lie algebra such that $\dim \fr{n}\geq p+3$, then $\Gamma$ is not isomorphic to the fundamental group of any smooth complex algebraic variety. 
\end{cor}
\begin{proof}
Assume that $\Gamma$ is isomorphic to a fundamental group of a any smooth complex algebraic variety. By \cref{WH}, we see that $\fr{n}$ admits a grading satisfying {\bf (W)} and {\bf (H)}. By \cref{pfiliformMMHS}, we have $p+3\leq \dim\fr{n}<p+3$. This contradicts our assumption. Thus we conclude that $\fr{n}$ does not admit such a grading. Hence, by \cref{fundamentallattice}, the Corollary holds. 
\end{proof}
\begin{rem}
According to the \cref{pfiliformMMHS}, it immediately follows that the Lie algebra $\fr{n}_m^{\bullet}=\ang{X_0,X_1,X_2,X_3,Y_1,\dots,Y_{m-4}}$ given in \cref{m-3fili} does not admit a grading satisfying {\bf (W)} and {\bf (H)} for $m\geq 5$ because $\fr{n}_m^{\bullet}$ is $(m-3)$-fliform. However, since $\fr{n}_5^2$ in \cref{m-2fili} is $2$-step, we cannot conclude immediately whether the Lie algebra $\fr{n}_5^2$ admits a grading such that  {\bf (W)} and {\bf (H)} hold.
\end{rem}
\begin{rem}\label{Cor501}
Let $\fr{n}$ be a $p$-filiform Lie algebra. Assume that $\fr{g}:=\fr{n}\o\bb{C}^n$ admits a grading $\fr{g}=\bo_{i\leq -1}\fr{n}_{i}$ which satisfies {\bf (W)} and {\bf (H)}. Since $\fr{g}$ is $(p+n)$-filiform Lie algebra, we have
\ali{(p+3)+n\leq \dim \fr{g}<p+n+3}
by \cref{pfiliformMMHS}. This contradicts the assumption.
Thus, we conclude that any direct sum decomposition of a $p$-filiform Lie algebra $\fr{n}$ such that $\dim\fr{n}\geq p+3$ and an abelian Lie algebra $\bb{C}^n$, $\fr{g}:=\fr{n}\o \bb{C}^n$ does not admit a grading satisfying {\bf (W)} and {\bf (H)}. This means that the direct product of the abelian group $\bb{Z}^n$ and any lattice $\Gamma$ in a simply-connected Lie group whose Lie algebra is $p$-filiform such that the dimension is greater than or equal to $p+3$, $\Gamma \times \bb{Z}^n$ is not isomorphic to the fundamental group of any smooth complex algebraic variety for all $n\geq 1$.  
\end{rem}
\begin{rem}
Let $\fr{n}=\bo_{i\in\bb{Z}}\fr{n}_i$ be a graded Lie algebra. The condition {\bf (H)} is satisfied by reindexing as $\fr{n}_{2i}':=\fr{n}_i$ for all $i\in\bb{Z}$, so that $\fr{n}=\bo_{i\in\bb{Z}}\fr{n}_{2i}'.$ Thus if we focus on only one of either {\bf (W)} or {\bf (H)}, then it becomes important to consider whether a nilpotent Lie algebra admits a grading satisfying {\bf (W)}.
\end{rem}
\section{Examples, Remarks and Low dimensional cases}
\subsection{A class of explicit counterexamples obtained by filiform Lie algebras}\label{filieg}
\begin{eg}\label{filieg}
In \cite{Nil}, some important example of filiform Lie algebras are given. In this section, we give some explicit examples which cannot be isomorphic to fundamental groups of smooth complex algebraic varieties obtained by some impotant example of filiform Lie algebras.
For integers $n\geq 3$ and $m\geq 5$, we define the Lie algebras $\mathcal{L}_n$, $\mathcal{R}_m$ and $\mathcal{Q}_{2m}$ by
\ali{\mathcal{L}_n=\ang{X_1,\dots,X_n}:&[X_1,X_i]=X_{i+1}&&(2\leq i\leq n-1)\\
\mathcal{R}_m=\ang{X_1,\dots,X_m}:&[X_1,X_i]=X_{i+1}&&(2\leq i\leq m-1)\\
&[X_2,X_j]=X_{j+2}&&(3\leq j\leq m-2)\\
\mathcal{Q}_{2m}=\ang{X_1,\dots,X_{2m}}:&[X_1,X_i]=X_{i+1}&&(2\leq i\leq 2m-1)\\
&[X_j,X_{2m-j+1}]=(-1)^{j+1}X_m&&(1\leq j\leq m).}
By \cref{mainthm}, these Lie algebras does not admit a gradings satisfying {\bf (W)} and {\bf (H)}. Thus, by \cref{Cor503} and \cref{Cor501}, we deduce that the groups $L_n\times \bb{Z}^l$, $R_m\times \bb{Z}^l$ and $Q_{2m}\times \bb{Z}^l$ are not isormorphic to fundamental groups of smooth complex algebraic varieties for any $n\geq 4$, $m\geq 5$ and $l\geq 0$ where the groups $L_n$, $R_m$ and $Q_{2m}$ are presented as follows:
\ali{
L_n&:=\left \langle x_1,\dots,x_n \middle\vert
\begin{array}{l}
[x_1, x_i]=x_{i+1}\ (1\leq i\leq n-1)\\
1=[x_1, x_n]=[x_j, x_k]\ (2\leq j<k\leq n)
\end{array}
\right \rangle \\
R_m&:= \left \langle x_1,\dots,x_m \middle\vert
\begin{array}{l}
[x_1, x_i]=x_{i+1}, [x_2, x_j]=x_{j+2}\\
(2\leq i\leq m-1, 3\leq j\leq m-2)\\
1=[x_1, x_m]=[x_2, x_{m-1}]=[x_2, x_m]=[x_k, x_l]\\
 (3\leq k<l\leq m)
\end{array}
\right \rangle \\
Q_{2m}&:= \left \langle x_1,\dots,x_{2m} \middle\vert
\begin{array}{l}
[x_1, x_i]=x_{i+1}, [x_j, x_{2m-j+1}]=x_{2m}^{(-1)^{j+1}}\\
(2\leq i\leq 2m-1, 1\leq j\leq m)\\
1=[x_1, x_m]=[x_k, x_l]\ (2\leq k<l\leq 2m-1, k+l\neq 2m+1)
\end{array}
\right \rangle.
}
\end{eg}
\subsection{The Low-dimensional case}
\subsubsection{Examples up to dimension $5$}\label{5dim}
\begin{rem}\label{L58}
The converse of the main theorem does not hold. There is a $2$-step nilpotent Lie algebra of dimension five which do not admit Morgan's mixed Hodge structures. The corresponding groups are the following:
\ali{\Gamma_{5,8}:= \left \langle x_1,x_2,x_3,x_4,x_5 \middle\vert
\begin{array}{l}
[x_1, x_2]=x_{3},[x_1,x_4]=x_5\\
1=[x_i,x_j]\ (\text{ for other }i<j)
\end{array}
\right \rangle
}We explain the outline of the proof of non-existence of Morgan's mixed Hodge structure on the Lie algebra $L_{5,8}$(Lie brackets are in {\sc Table1}). The Lie algebra $L_{5,8}$ is not $p$-filiform for any $p$. However, in the same way as \cref{mainthm}, we can show that $L_{5,8}$ does not admit a grading satisfying the conditions {\bf (W)} and {\bf (H)}. We provide an outline of the proof as follows:

If there exists a grading $L_{5,8}=\fr{n}_{i_1}\o \fr{n}_{i_2}\o\cdots \o\fr{n}_{i_k}$ such that the conditions {\bf (W)} and {\bf (H)} hold, then {\bf (W)} implies that $i_1=-1$ or $-2$. Suppose that $i_1=-1.$ By {\bf (H)}, it follows that $\dim \fr{n}_{-1}=2$ because $\dim H^1(L_{5,8})=3.$ By {\bf (W)}, there exists an element $0\neq Y_3\in \fr{n}_{-2}$ such that $Y_3\notin C^1\fr{n}$. Let $\fr{n}_{-1}=\ang{Y_1,Y_2}.$ If $[Y_1,Y_2]\neq 0$, then, by {\bf (H)}, we have $\ang{[Y_1,Y_2],Y_3}=\fr{n}_{-2}$, $0\neq [[Y_1,Y_2],Y_3]\in \fr{n}_{-4}$ and $[Y_i,Y_3]=0=[Y_i,[Y_1,Y_2]]$ for $i=1,2$.

Thus we obtain the grading $L_{5,8}=\ang{Y_1,Y_2}_{-1}\o\ang{Y_3,[Y_1,Y_2]}_{-2}\o\ang{[[Y_1,Y_2],Y_3]}_{-4}.$
Since $L_{5,8}$ is $2$-step, we have $[[Y_1,Y_2],Y_3]=0$. This contradicts the fact that $\dim L_{5,8}=5.$ Thus $[Y_1,Y_2]=0$ and $\fr{n}_{-2}=\ang{Y_3}.$ Since $L_{5,8}$ is non-abelian, it follows that $Y_4=[Y_1,Y_3]\neq 0$ and $Y_5=[Y_2,Y_3]\neq 0.$ By {\bf (H)}, we have $\fr{n}_{-3}=\ang{Y_4,Y_5}$. Let $y_1,y_2,y_3,y_4$ and $y_5$ be the dual basis of $Y_1,Y_2,Y_3,Y_4$ and $Y_5$. Then we have $d(y_3\w y_5)=0$ and $H_5^2\neq \0 .$ Thus we conclude that $i_1= -2.$ In this case, we can show existence of a non-zero element in $H_6^2$ or $H_8^2.$  
\end{rem}
\begin{eg}
We consider the gradings of $p$-filiform Lie algebras of dimension up to $5$ for $p=1,2$. We use the classification given in  \cite{nilpotent}. We omit the decomposable Lie algebras, i.e. the Lie algebras which are direct sums of proper ideals. {\sc Table $1$} shows the classification of the indecomposable nilpotent Lie algebras of dimension up to $5$. {\sc Table $2$} summarizes whether the nilpotent Lie algebras of dimension up to $5$ satisfy only condition {\bf (W)} or both {\bf (W)} and {\bf (H)}. The examples of gradings satisfying each condition are provided as well. By \cref{trivialextension} and \cref{Cor501}, we do not need to consider a grading of $\fr{n}\o\bb{C}^n$ for any $\fr{n}$ which admits a grading satisfying {\bf (W)} and {\bf (H)}. 
\end{eg}
\begin{prob}
It is well-known that $\bb{C}-\0$, $(\bb{C}-\0)\times (\bb{C}-\0)$ and $(H_3(\bb{R})/\Gamma)\times \bb{R}_{>0}$ have structure of smooth complex algebraic varieties and Malcev completions of its fundamental groups are isomorphic to respectively $L_{1,1}$, $L_{2,1}$ and $L_{3,2}$ where $\Gamma$ is a lattice in the Heisenberg group $H_3(\bb{R}).$ Can we construct a any smooth complex algebraic variety $M$ such that the fundamental group $\pi_1(M,x)$ is a lattice in a simply connected nilpotent Lie group whose Lie algebra is isomorphic to $L_{5,9}$?
\end{prob}
By the above examples and \cref{Cor501}, we can state that the groups $\Gamma_{4,3}\times \bb{Z}^n$, $\Gamma_{5,5}\times \bb{Z}^n$,$\Gamma_{5,6}\times \bb{Z}^n$, $\Gamma_{5,7}\times \bb{Z}^n$ and $\Gamma_{5,8}\times \bb{Z}^n$ are not isomorphic to fundamental groups of smooth complex algebraic varieties for any integer $n\geq 0$ where $\Gamma_{\bullet,\bullet}$ is presented as follows:
\ali{
\Gamma_{4,3}&:= \left \langle x_1,x_2,x_3,x_4 \middle\vert
\begin{array}{l}
[x_1, x_i]=x_{i+1}\ (i=2,3)\\
1=[x_1, x_4]=[x_j, x_k]\ (2\leq j<k\leq 4)
\end{array}
\right \rangle \\
\Gamma_{5,5}&:= \left \langle x_1,x_2,x_3,x_4,x_5 \middle\vert
\begin{array}{l}
[x_1, x_3]=x_{4},[x_1,x_4]=x_5, [x_3,x_2]=x_5\\
1=[x_1, x_2]=[x_1,x_5]=[x_2, x_4]=[x_2,x_5]=[x_i,x_j]\ (3\leq i<j\leq 5)
\end{array}
\right \rangle \\
\Gamma_{5,6}&:= \left \langle x_1,x_2,x_3,x_4,x_5 \middle\vert
\begin{array}{l}
[x_1, x_i]=x_{i+1}\ (i=2,3,4),[x_2,x_3]=x_5\\
1=[x_1,x_5]=[x_2, x_4]=[x_2,x_5]=[x_i,x_j]\ (3\leq i<j\leq 5)
\end{array}
\right \rangle \\
\Gamma_{5,7}&:= \left \langle x_1,x_2,x_3,x_4,x_5 \middle\vert
\begin{array}{l}
[x_1, x_i]=x_{i+1}\ (i=2,3,4)\\
1=[x_1,x_5]=[x_i,x_j]\ (2\leq i<j\leq 5)
\end{array}
\right \rangle \\
\Gamma_{5,8}&:= \left \langle x_1,x_2,x_3,x_4,x_5 \middle\vert
\begin{array}{l}
[x_1, x_2]=x_{3},[x_1,x_4]=x_5\\
1=[x_1, x_3]=[x_1,x_5]=[x_i,x_j]\ (2\leq i<j\leq 5)
\end{array}
\right \rangle
.}
\\

\subsubsection{The six-dimensional case}\label{6dim}
\begin{rem}
Using the classification of nilpotent Lie algebras of dimension up to $6$(see \cite{Nil},\cite{Grnil} or \cite{classnil}), we give a sufficient condition for finitely generated torsion-free nilpotent groups corresponding to  $6$-dimensional Lie algebras not to be isomorphic to the fundamental groups of smooth complex algebraic varieties. {\sc Table $3$} shows that the classification of $\bb{C}$-nilpotent Lie algebras of dimension $6$. As in {\sc Table 2}, {\sc Table 4} compares gradings that satisfy conditions {\bf (W)} and {\bf (H)} in the six-dimensional case.

The Lie algebras $L_{6,19}$ through $L_{6,28}$ are not $p$-filiform for any $p$. Thus we have to consider whether they admit certain gradings individually. We can explicitly construct such gradings on $L_{6,21}(-1)$, $L_{6,22}(0)$, $L_{6,22}(1)$, $L_{6,24}(0)$ and $L_{6,24}(1)$ (see {\sc Table $4$}). We outline a proof of non-existence of Morgan's mixed Hodge structure on the remaining Lie algebras.  We can show that $L_{6,26}$ does not admit Morgan's mixed Hodge structure in the same way as in \cref{L58}. Moreover, if a $6$-dimensional $3$-step nilpotent Lie algebra $\fr{n}$ admits a grading such that {\bf (W)} and {\bf (H)} hold and $\dim H^1(\fr{n})=3$, then it follows that the types of gradings on $\fr{n}$ can be classified as follows:
\ali{\fr{n}&=\ang{Y_1,Y_2}_{-1}\o\ang{Y_3}_{-2}\o\ang{Y_4,Y_5}_{-3}\o\ang{Y_6}_{i_4}\\&: [Y_1,Y_2]=Y_3,[Y_1,Y_3]=Y_4,[Y_2,Y_3]=Y_5,\fr{n}_{i_4}\neq \0,\\  
\fr{n}&=\ang{Y_1,Y_2,Y_3}_{-2}\o\ang{Y_4}_{-4}\o\ang{Y_5}_{-6}\o\ang{Y_6}_{i_4}\\&:Y_4=[Y_i,Y_j]\neq 0,Y_5=[Y_k,Y_4],\fr{n}_{i_4}\neq \0\text{ for some }i,j,k\in \{1,2,3\},\\
\fr{n}&=\ang{Y_1,Y_2,Y_3}_{-2}\o\ang{Y_4}_{-4}\o\ang{Y_5,Y_6}_{-6}\\
&:Y_4=[Y_i,Y_j]\neq 0,Y_5=[Y_k,Y_4],Y_6=[Y_l,Y_4]\neq 0\text{ for some }i,j,k,l\in \{1,2,3\},\\ 
\fr{n}&=\ang{Y_1,Y_2,Y_3}_{-2}\o\ang{Y_4,Y_5}_{-4}\o\ang{Y_6}_{i_3}\\
&:Y_4=[Y_i,Y_j]\neq 0,Y_5=[Y_k,Y_l]\fr{n}_{i_4}\neq \0\text{ for some }i,j,k,l\in \{1,2,3\}.} 
We see that there exists a non-zero element such that the degree is greater than $4$ in the second cohomology of $\fr{n}$ in the above cases. This contradicts the condition {\bf (W)}. Thus the Lie algebras $L_{6,19}$, $L_{6,23}$, $L_{6,25}$, $L_{6,26}$ and $L_{6,27}$ do not admit a gradings satisfying the conditions {\bf (W)} and {\bf (H)}. Finally, we consider the case of $L_{6,28}.$ Note that $\dim H^2(L_{6,28})=4.$ Suppose that $L_{6,28}$ has a grading satisfying {\bf (W)} and {\bf (H)}. Then we can show that the types of gradings can be classified as follows:
\ali{
\fr{n}&=\ang{Y_1,Y_2}_{i_1}\o\ang{Y_3}_{2i_2}\o\ang{Y_4,Y_5}_{3i_1}\o\ang{Y_6}_{i_4}\\
&: [Y_1,Y_2]=Y_3,[Y_1,Y_3]=Y_4,[Y_2,Y_3]=Y_5,\fr{n}_{i_4}\neq \0,\\    
\fr{n}&=\ang{Y_1,Y_2}_{i_1}\o\ang{Y_3}_{2i_2}\o\ang{Y_4}_{3i_1}\o\ang{Y_5,Y_6}_{i_4}\\
&: [Y_1,Y_2]=Y_3,[Y_i,Y_3]=Y_4,\fr{n}_{i_4}\neq \0\text{ and }\dim \fr{n}_{i_4}=2 \text{ for some }i\in\{1,2\},\\    
\fr{n}&=\ang{Y_1,Y_2}_{i_1}\o\ang{Y_3}_{2i_2}\o\ang{Y_4}_{3i_1}\o\ang{Y_5}_{i_4}\o\ang{Y_6}_{i_5}\\
&: [Y_1,Y_2]=Y_3,[Y_i,Y_3]=Y_4,\fr{n}_{i_4},\fr{n}_{i_5}\neq \0\text{ and }\dim \fr{n}_{i_4}=1=\dim\fr{n}_{i_5} \text{ for some }i,j\in\{1,2\}. 
}
Thus we have 
\ali{\bw ^2\fr{n}^*=\ang{y_1\w y_2}_{-2i_1}\o\ang{y_1\w y_3,y_2\w y_3}_{-3i_1}\o\ang{y_1\w y_4,y_1\w y_5,y_2\w y_4,y_2\w y_5}_{-4i_1}\o\cdots,}
\ali{\bw ^2\fr{n}^*=\ang{y_1\w y_2}_{-2i_1}\o\ang{y_1\w y_3,y_2\w y_3}_{-3i_1}\o\ang{y_1\w y_4,y_2\w y_4}_{-4i_1}\o\cdots}
By using the above gradings on $\bw^2\fr{n}^*$, we have $\dim H^2(\fr{n})\leq 3.$ This contradicts that $\dim H^2(\fr{n})=4.$ 
\end{rem}
The following theorem follows from the above remarks.
\begin{thm}\label{lowdim}
Let $N$ be a simply connected nilpotent Lie group and $\fr{n}$ its Lie algebra such that $\dim\fr{n}\leq 6$. Let $\Gamma$ be a lattice of $N$. If $\fr{n}$ is isomorphic to one of the following Lie algebras:
\ali{
\text{filiform}&:L_{4,3}, L_{5,a}\ (a=6,7),\ L_{6,b}\ (14\leq b\leq 18)\\
\text{$2$-filiform}&:L_{4,3}\o\bb{C},\ L_{5,5},\ L_{5,a}\o\bb{C},\ L_{6,c}\ (11\leq c\leq 13)\\
\text{$3$-filiform}&:L_{4,3}\o\bb{C}^2,\ L_{5,5}\o\bb{C},\ L_{6,10}\\
\text{otherwise}&:L_{5,8},\ L_{6,19}(-1),\ L_{6,20},\ L_{6,23},\ L_{6,d}\ (25\leq d\leq 28)
}
then $\Gamma$ is not isomorphic to the fundamental group of any smooth complex algebraic variety. 
\end{thm}
\newpage
\centering{{\bf {\sc Table $1$} Classification of $\bb{C}-$nilpotent Lie algebras in dimension up to $5$}
}
\begin{longtable}{|c|l|l|}

\toprule
          &   & Lie brackets  \\ \midrule
dim$=1$&$L_{1,1}=\bb{C}$          &abelian \\ \midrule
dim$=2$&$L_{2,1}=\bb{C}^2$          &abelian \\ \midrule
dim$=3$&$L_{3,1}=\bb{C}^3$          &abelian \\ 
&$L_{3,2}$          &$[X_1,X_2]=X_3$ \\ \midrule
dim$=4$&$L_{4,1}=\bb{C}^4$          &abelian \\ 
&$L_{4,2}=L_{3,2}\o\bb{C}$          &$[X_1,X_2]=X_3,[Y_4,X_i]=0\text{ for all }i$ \\ 
&$L_{4,3}$          &$[X_1,X_2]=X_3, [X_1,X_3]=X_4$ \\ \midrule
dim$=5$&$L_{5,1}=\bb{C}^4$          &abelian \\ 
&$L_{5,2}=L_{3,2}\o\bb{C}^2$          &$[X_1,X_2]=X_3,[Y_4,X_i]=0=[Y_5,X_i]=[Y_4,Y_5]\text{ for all }i$ \\ 
&$L_{5,3}=L_{4,3}\o\bb{C}$          &$[X_1,X_2]=X_3, [X_1,X_3]=X_4, [Y_5,X_i]=0\text{ for all }i$ \\
&$L_{5,4}$          &$[X_4,X_1]=X_5, [X_2,X_3]=X_5$ \\ 
&$L_{5,5}$          &$[X_1,X_3]=X_4, [X_1,X_4]=X_5,[X_3,X_2]=X_5$ \\ 
&$L_{5,6}$          &$[X_1,X_2]=X_3, [X_1,X_3]=X_4,[X_1,X_4]=X_5,[X_2,X_3]=X_5$ \\ 
&$L_{5,7}$          &$[X_1,X_2]=X_3, [X_1,X_3]=X_4,[X_1,X_4]=X_5$ \\ 
&$L_{5,8}$          &$[X_1,X_2]=X_3, [X_1,X_4]=X_5$ \\ 
&$L_{5,9}$          &$[X_1,X_2]=X_3, [X_2,X_3]=X_4,[X_1,X_3]=X_5$ \\ \bottomrule
\end{longtable}
\centering{{\bf {\sc Table $2$} Comparing the condition {\bf (W)} and {\bf (H)} in dimension up to $5$}}
\begin{longtable}{|l|c|c|l|l|}
\toprule
          &\hspace{-0.cm}{\bf (W)} & {\bf (W) and (H)} & grading satisfying {\bf (W)} &{\bf (W)} and {\bf (H)} \\ \midrule
$L_{1,1}=\bb{C}=\ang{X}$ &$\bigcirc$            & $\bigcirc$                                                          & $\ang{X}_{-2}$ &Same as {\bf (W)} \\ \midrule
$L_{2,1}=\bb{C}^2$ &$\bigcirc$            & $\bigcirc$                                                          & $\ang{X_1,X_2}_{-1}$ & Same as {\bf (W)}\\ \midrule
$L_{3,1}=\bb{C}^3$ &$\bigcirc$            & $\bigcirc$                                                          & $\ang{X_1,X_2}_{-1}\o\ang{X_3}_{-2}$ &Same as {\bf (W)} \\ 
$L_{3,2}$(filiform) &$\bigcirc$            & $\bigcirc$                                                          & $\ang{X_1,X_2}_{-1}\o \ang{X_3}_{-2}$ & Same as {\bf (W)} \\ \midrule
$L_{4,1}=\bb{C}^4$ &$\bigcirc$            & $\bigcirc$                                                          & $\ang{X_1,X_2,X_3,X_4}_{-1}$ &Same as {\bf (W)} \\ 
$L_{4,2}$($2$-filiform) &$\bigcirc$            & $\bigcirc$                                                          & $\ang{X_1,X_2}_{-1}\o \ang{X_3,Y_4}_{-2}$ & Same as {\bf (W)} \\
$L_{4,3}$(filiform) & $\bigcirc$     & $\bigtimes$                                                          & $\ang{X_1,X_2}_{-1}\o \ang{X_3}_{-2}\o \ang{X_4}_{-3}$ &  \\ \midrule
$L_{5,1}=\bb{C}^4$ &$\bigcirc$            & $\bigcirc$                                                          & $\ang{X_1,X_2,X_3,X_4}_{-1}\o\ang{X_5}_{-2}$ &Same as {\bf (W)} \\ 
$L_{5,2}$($3$-filiform) &$\bigcirc$            & $\bigcirc$                                                          & $\ang{X_1,X_2}_{-1}\o \ang{X_3,Y_4,Y_5}_{-2}$ & Same as {\bf (W)} \\
$L_{5,3}$($2$-filiform) & $\bigcirc$     & $\bigtimes$                                                          & $\ang{X_1,X_2}_{-1}\o \ang{X_3,Y_5}_{-2}\o \ang{X_4}_{-3}$ &  \\ 
$L_{5,4}$    & $\bigcirc$                           &  $\bigcirc$                                                         & $\ang{X_1,X_2,X_3,X_4}_{-1}\o \ang{X_5}_{-2}$ & Same as {\bf (W)} \\ 
$L_{5,5}$($2$-filiform)     & $\bigcirc$                           &$\bigtimes$                                                           &$\ang{X_1,X_3}_{-1}\o\ang{X_2,X_4}_{-2}\o\ang{X_5}_{-3}$  &  \\
$L_{5,6}$(filiform)    & $??$                           &$\bigtimes$                                                           &  &  \\
$L_{5,7}$(filiform)    & $??$                           &$\bigtimes$                                                           &  &  \\
$L_{5,8}$     & $\bigcirc$                           &$\bigtimes$                                                           & $\ang{X_1,X_2,X_4}_{-1}\o\ang{X_3,X_5}_{-2}$ &  \\
$L_{5,9}$    & $\bigcirc$                           &$\bigcirc$                                                           & $\ang{X_1,X_2}_{-1}\o \ang{X_3}_{-2}\o \ang{X_4,X_5}_{-3}$ & Same as {\bf (W)}\\ \bottomrule
\end{longtable}
 \newpage
\centering{{\bf {\sc Table $3$} Classification of $\bb{C}-$nilpotent Lie algebras in dimension $6$}
}
\begin{longtable}{|c|l|}

\toprule
             & Lie brackets  \\ \midrule
$L_{6,1}=\bb{C}^6$          &abelian \\ 
$L_{6,2}=L_{3,2}\o\bb{C}^3$          &$[X_1,X_2]=X_3,[Y_4,X_i]=0=[Y_5,X_i]=[Y_6,X_i]=[Y_j,Y_k]\text{ for all }i\text{ and }4\leq j<k\leq 6$ \\ 
$L_{6,3}=L_{4,3}\o\bb{C}^2$          &$[X_1,X_2]=X_3, [X_1,X_3]=X_4, [Y_5,X_i]=0=[Y_6,X_i]\text{ for all }i$ \\
$L_{6,4}=L_{5,4}\o\bb{C}$          &$[X_4,X_1]=X_5, [X_2,X_3]=X_5,[Y_6,X_i]=0\text{ for all }i$ \\ 
$L_{6,5}=L_{5,5}\o\bb{C}$          &$[X_1,X_3]=X_4, [X_1,X_4]=X_5,[X_3,X_2]=X_5,[Y_6,X_i]=0\text{ for all }i$ \\ 
$L_{6,6}=L_{5,6}\o\bb{C}$          &$[X_1,X_2]=X_3, [X_1,X_3]=X_4,[X_1,X_4]=X_5,[X_2,X_3]=X_5,[Y_6,X_i]=0\text{ for all }i$ \\ 
$L_{6,7}=L_{5,7}\o\bb{C}$          &$[X_1,X_2]=X_3, [X_1,X_3]=X_4,[X_1,X_4]=X_5,[Y_6,X_i]=0\text{ for all }i$ \\ 
$L_{6,8}=L_{5,8}\o\bb{C}$          &$[X_1,X_2]=X_3, [X_1,X_4]=X_5,[Y_6,X_i]=0\text{ for all }i$ \\ 
$L_{6,9}=L_{5,9}\o\bb{C}$          &$[X_1,X_2]=X_3, [X_2,X_3]=X_4,[X_1,X_3]=X_5,[Y_6,X_i]=0\text{ for all }i$ \\ 
$L_{6,10}$          &$[X_2,X_3]=X_4, [X_5,X_1]=X_6, [X_2,X_4]=X_6$ \\ 
$L_{6,11}$          &$[X_1,X_2]=X_3, [X_1,X_3]=X_5,[X_1,X_5]=X_6,[X_2,X_3]=X_6,[X_2,X_4]=X_6$ \\ 
$L_{6,12}$          &$[X_2,X_3]=X_4, [X_2,X_4]=X_5,[X_3,X_1]=X_6,[X_2,X_5]=X_6$ \\ 
$L_{6,13}$          &$[X_1,X_3]=X_4, [X_1,X_4]=X_5,[X_3,X_2]=X_5,[X_1,X_5]=X_6,[X_4,X_2]=X_6$ \\ 
$L_{6,14}$          &$[X_1,X_2]=X_3, [X_1,X_3]=X_4,[X_1,X_4]=X_5,[X_2,X_3]=X_5,[X_2,X_5]=X_6,[X_4,X_3]=X_6$ \\ 
$L_{6,15}$          &$[X_1,X_2]=X_3, [X_1,X_3]=X_4,[X_1,X_4]=X_5,[X_2,X_3]=X_5,[X_2,X_5]=X_6,[X_2,X_4]=X_6$ \\ 
$L_{6,16}$          &$[X_1,X_2]=X_3, [X_1,X_3]=X_4,[X_1,X_4]=X_5,[X_2,X_5]=X_6,[X_4,X_3]=X_6$ \\ 
$L_{6,17}$          &$[X_2,X_1]=X_3, [X_2,X_3]=X_4,[X_2,X_4]=X_5,[X_1,X_3]=X_6,[X_2,X_5]=X_6$ \\ 
$L_{6,18}$          &$[X_1,X_2]=X_3, [X_1,X_3]=X_4,[X_1,X_4]=X_5,[X_1,X_5]=X_6$ \\ 
$L_{6,19}(-1)$          &$[X_1,X_2]=X_3, [X_1,X_4]=X_5,[X_2,X_5]=X_6,[X_4,X_3]=X_6$ \\ 
$L_{6,20}$          &$[X_1,X_2]=X_3, [X_1,X_4]=X_5,[X_1,X_5]=X_6,[X_2,X_3]=X_6$ \\ 
$L_{6,21}(-1)$          &$[X_1,X_2]=X_3, [X_2,X_3]=X_4,[X_1,X_3]=X_5,[X_1,X_4]=X_6,[X_2,X_5]=X_6$ \\ 
$L_{6,22}(0)$          &$[X_2,X_4]=X_5, [X_4,X_1]=X_6,[X_2,X_3]=X_6$ \\ 
$L_{6,22}(1)$          &$[X_1,X_2]=X_3, [X_4,X_5]=X_6$ \\ 
$L_{6,23}$          &$[X_1,X_2]=X_3, [X_1,X_4]=X_5,[X_1,X_5]=X_6,[X_4,X_2]=X_6$ \\ 
$L_{6,24}(0)$          &$[X_1,X_3]=X_4, [X_3,X_4]=X_5,[X_1,X_4]=X_6,[X_3,X_2]=X_6$ \\ 
$L_{6,24}(1)$          &$[X_1,X_2]=X_3, [X_2,X_3]=X_5,[X_2,X_4]=X_5,[X_1,X_3]=X_6$ \\ 
$L_{6,25}$          &$[X_1,X_2]=X_3, [X_1,X_3]=X_4,[X_1,X_5]=X_6$ \\ 
$L_{6,26}$          &$[X_1,X_2]=X_3, [X_2,X_4]=X_5,[X_1,X_4]=X_6$ \\ 
$L_{6,27}$          &$[X_1,X_2]=X_3, [X_1,X_3]=X_4,[X_2,X_5]=X_6$ \\ 
$L_{6,28}$          &$[X_1,X_2]=X_3, [X_2,X_3]=X_4,[X_1,X_3]=X_5,[X_1,X_5]=X_6$ \\ 
\bottomrule
\end{longtable}
\newpage
\begin{table}
\centering{{\bf {\sc Table $4$} Comparing the condition {\bf (W)} and {\bf (H)} in dimension $6$}}
\begin{tabular}{|l|c|c|l|l|}
\toprule
          &\hspace{-0.cm}{\bf (W)} & {\bf (W) and (H)} & grading satis. {\bf (W)} &{\bf (W)} and {\bf (H)} \\ \midrule
$L_{6,1}=\bb{C}^4$ &$\bigcirc$            & $\bigcirc$                                                          & $\ang{X_1,X_2,X_3,X_4}_{-1}\o\ang{X_5,X_6}_{-2}$ &Same as {\bf (W)} \\ 
$L_{6,2}$($4$-filiform) &$\bigcirc$            & $\bigcirc$                                                          & $\ang{X_1,X_2}_{-1}\o \ang{X_3,Y_4,Y_5,Y_6}_{-2}$ & Same as {\bf (W)} \\
$L_{6,3}$($3$-filiform) & $\bigcirc$     & $\bigtimes$                                                          & $\ang{X_1,X_2}_{-1}\o \ang{X_3,Y_5,Y_6}_{-2}\o \ang{X_4}_{-3}$ &  \\ 
$L_{6,4}$    & $\bigcirc$                           &  $\bigcirc$                                                         & $\ang{X_1,X_2,X_3,X_4}_{-1}\o \ang{X_5,Y_6}_{-2}$ & Same as {\bf (W)} \\ 
$L_{6,5}$($3$-filiform)     & $\bigcirc$                           &$\bigtimes$                                                           &$\ang{X_1,X_3}_{-1}\o\ang{X_2,X_4,Y_6}_{-2}\o\ang{X_5}_{-3}$  &  \\
$L_{6,6}$($2$-filiform)    & $??$                           &$\bigtimes$                                                           &  &  \\
$L_{6,7}$($2$-filiform)    & $??$                           &$\bigtimes$                                                           &  &  \\
$L_{6,8}$     & $\bigcirc$                           &$\bigtimes$                                                           & $\ang{X_1,X_2,X_4}_{-1}\o\ang{X_3,X_5,Y_6}_{-2}$ &  \\
$L_{6,9}$    & $\bigcirc$                           &$\bigcirc$                                                           & $\ang{X_1,X_2}_{-1}\o \ang{X_3,Y_6}_{-2}\o \ang{X_4,X_5}_{-3}$ & Same as {\bf (W)}\\
$L_{6,10}$($3$-filiform) &\hspace{0.00cm}$\bigcirc$            & $\bigtimes$                                                          & $\ang{X_1,X_2,X_3}_{-1}\o\ang{X_4,X_5}_{-2}\o\ang{X_6}_{-3}$ & \\
$L_{6,a}$($2$-filiform, $a=11,12,13$) &\hspace{0.00cm}$??$            & $\bigtimes$& &\\
$L_{6,b}$(filiform,$14\leq b\leq 18$) &\hspace{0.00cm}$??$            & $\bigtimes$                                                          &  & \\
$L_{6,19}(-1)$ &\hspace{0.00cm}$\bigcirc $            & $\bigtimes $                                                          &$\ang{X_1,X_2,X_3}_{-1}\o\ang{X_3,X_5}_{-2}\o\ang{X_6}_{-3}$  & \\
$L_{6,20}$ &\hspace{0.00cm}$\bigcirc$            & $\bigtimes $                                                          &$\ang{X_1,X_2,X_4}_{-1}\o \ang{X_3,X_5}_{-2}\o\ang{X_6}_{-3}$  & \\
$L_{6,21}(-1)$ &\hspace{0.00cm}$\bigcirc$            & $\bigcirc$                                                          &$\ang{X_1,X_2}_{-1}\o\ang{X_3}_{-2}\o\ang{X_4,X_5}_{-3}\o\ang{X_6}_{-4}$  &Same as {\bf (W)} \\
$L_{6,22}(0)$ &\hspace{0.00cm}$\bigcirc$            & $\bigcirc$                                                          &$\ang{X_1,X_2,X_3,X_4}_{-1}\o\ang{X_5,X_6}_{-2}$  &Same as {\bf (W)} \\
$L_{6,22}(1)$ &\hspace{0.00cm}$\bigcirc$            &$\bigcirc$                                                          &$\ang{X_1,X_2,X_4,X_5}_{-1}\o\ang{X_3,X_6}_{-2}$  &Same as {\bf (W)}  \\
$L_{6,23}$ &\hspace{0.00cm}$??$            & $\bigtimes$                                                          &  & \\
$L_{6,24}(0)$ &\hspace{0.00cm}$\bigcirc$            & $\bigcirc$                                                          &$\ang{X_1,X_3}_{-1}\o \ang{X_2,X_4}_{-2}\o\ang{X_5,X_6}_{-3}$  &Same as {\bf (W)} \\
$L_{6,24}(1)$ &\hspace{0.00cm}$\bigcirc$            & $\bigcirc$                                                          &$\ang{X_1,X_2}_{-1}\o \ang{X_3,X_4}_{-2}\o\ang{X_5,X_6}_{-3}$  &Same as {\bf (W)} \\
$L_{6,25}$ &\hspace{0.00cm}$\bigcirc$            & $\bigtimes$                                                          &$\ang{X_1,X_2,X_5}_{-1}\o\ang{X_3,X_6}_{-2}\o\ang{X_4}_{-3}$  & \\
$L_{6,26}$ &\hspace{0.00cm}$\bigcirc$            & $\bigtimes$                                                          & $\ang{X_1,X_2,X_4}_{-1}\o\ang{X_3,X_5,X_6}_{-2}$ & \\
$L_{6,27}$&\hspace{0.00cm}$\bigcirc$            & $\bigtimes$                                                          & $\ang{X_1,X_2,X_5}_{-1}\o\ang{X_3,X_6}_{-2}\o\ang{X_4}_{-3}$ & \\
$L_{6,28}$&\hspace{0.00cm}$??$            & $\bigtimes$                                                          &  & \\
\bottomrule
\end{tabular}
\end{table}

\end{document}